\documentclass[11pt, a4paper]{article}
\usepackage{amssymb,amsmath,amsthm,tocloft}
\usepackage{fancyhdr}
\usepackage{comment, enumerate}
\usepackage[british]{babel}
\usepackage{stackengine}
\usepackage[usenames,dvipsnames,table]{xcolor}
\usepackage{graphicx,tikz,caption,subcaption}
\usetikzlibrary{decorations.pathreplacing}
\usetikzlibrary{decorations.pathmorphing}
\usetikzlibrary{calc,arrows,decorations.markings}
\tikzset{snake it/.style={decorate, decoration=snake}}
\usepackage{bbm,dsfont}

\usepackage{hyperref}

\usepackage[margin=2.8cm]{geometry}

\newtheorem{theorem}{Theorem}[section]

\newtheorem{lemma}[theorem]{Lemma}
\newtheorem{cor}[theorem]{Corollary}

\newtheorem{proposition}[theorem]{Proposition}
\newtheorem{prob}[theorem]{Problem}

\newcommand{\repeatlabel}{}
\newtheorem*{repeatlemma}{Lemma \repeatlabel}

\theoremstyle{definition}

\newtheorem{rmk}[theorem]{Remark}

\newtheorem{claim}[theorem]{Claim}

\newtheorem{question}[theorem]{Question}

\newenvironment{poc}{\begin{proof}[Proof of claim]}{\end{proof}}

\newcommand{\bN}{\mathbb{N}}

\newcommand{\cH}{\mathcal{H}}

\newcommand{\cF}{\mathcal{F}}

\title{Fractional Helly theorem for Cartesian products of convex sets}
\author{Debsoumya Chakraborti\thanks{
Discrete Mathematics Group, Institute for Basic Science (IBS), South Korea.
E-mail: {\tt \{debsoumya, jinhakim, minkikim\}@ibs.re.kr}. Debsoumya Chakraborti, Jinha Kim, and Minki Kim were supported by the Institute for Basic Science (IBS-R029-C1).}\and 
Jaehoon Kim\thanks{Department of Mathematical Sciences, KAIST, South Korea. E-mail: {\tt jaehoon.kim@kaist.ac.kr}. Jaehoon Kim was supported by the POSCO Science Fellowship of POSCO TJ Park Foundation and by the KAIX Challenge program of KAIST Advanced Institute for Science-X.}\and
Jinha Kim\footnotemark[1]\and
Minki Kim\footnotemark[1]~\thanks{Corresponding author.}\and
Hong Liu\thanks{Mathematics Institute, University of Warwick, UK. E-mail: {\tt h.liu.9@warwick.ac.uk}. Hong Liu was supported by the UK Research and Innovation Future Leaders Fellowship MR/S016325/1.} }
\date{\today}

\begin{document}
\maketitle

\begin{abstract}
Helly's theorem and its variants show that for a family of convex sets in Euclidean space, local intersection patterns influence global intersection patterns. A classical result of Eckhoff in 1988 provided an optimal fractional Helly theorem for axis-aligned boxes, which are Cartesian products of line segments. Answering a question raised by B\'ar\'any and Kalai, and independently Lew,  we generalize Eckhoff's result to Cartesian products of convex sets in all dimensions.
   
In particular, we prove that given $\alpha \in (1-\frac{1}{t^d},1]$ and a finite family $\mathcal{F}$ of Cartesian products of convex sets $\prod_{i\in[t]}A_i$ in $\mathbb{R}^{td}$ with $A_i\subset \mathbb{R}^d$, if at least $\alpha$-fraction of the $(d+1)$-tuples in $\mathcal{F}$ are intersecting, then at least $(1-(t^d(1-\alpha))^{1/(d+1)})$-fraction of sets in $\mathcal{F}$ are intersecting. This is a special case of a more general result on intersections of $d$-Leray complexes. We also provide a construction showing that our result on $d$-Leray complexes is optimal. Interestingly, the extremal example is representable as a family of cartesian products of convex sets, implying that the bound $\alpha>1-\frac{1}{t^d}$ and the fraction $(1-(t^d(1-\alpha))^{1/(d+1)})$ above are also best possible.
   
The well-known optimal construction for fractional Helly theorem for convex sets in $\mathbb{R}^d$ does not have $(p,d+1)$-condition for sublinear $p$, that is, it contains a linear-size subfamily with no intersecting $(d+1)$-tuple. Inspired by this, we give constructions showing that, somewhat surprisingly, imposing additional $(p,d+1)$-condition has negligible effect on improving the quantitative bounds in neither the fractional Helly theorem for convex sets nor Cartesian products of convex sets. Our constructions offer a rich family of distinct extremal configurations for fractional Helly theorem, implying in a sense that the optimal bound is stable.
\end{abstract}

\section{Introduction}
A family of non-empty sets is {\em intersecting} if all sets within have an element in common. Let $\mathcal{F}$ be a (possibly infinite) family of non-empty sets.
The {\em Helly number} of $\mathcal{F}$ is the minimal size of a subfamily $\cH$ such that every proper subfamily of $\cH$ is intersecting but $\mathcal{H}$ itself is not intersecting.
Helly's theorem~\cite{Hel23}, one of the most classical result about intersection patterns of convex sets in Euclidean spaces, asserts that the family of all convex sets in $\mathbb{R}^d$ has Helly number $d+1$.

There are a large number of variants and applications of Helly's theorem.
See \cite{ALS17} for an overview on such Helly type theorems.
One of the most important generalizations of Helly's theorem is the fractional Helly theorem, showing that if we only assume a positive fraction of the $(d+1)$-tuples are intersecting, then there is still a large intersecting subfamily. 

More precisely, the fractional Helly theorem asserts that for every positive integer $d$, there exists a function $\beta_d:(0,1]\to(0,1]$ such that for every $\alpha \in (0,1]$ and finite family $\mathcal{F}$ of convex sets in $\mathbb{R}^d$, if at least $\alpha\binom{|\mathcal{F}|}{d+1}$ of the $(d+1)$-tuples of $\mathcal{F}$ are intersecting, then $\mathcal{F}$ contains an intersecting subfamily of size at least $\beta_d(\alpha)|\mathcal{F}|$.
The fractional Helly theorem was first shown by Katchalski and Liu~\cite{KL79} with a lower bound $\beta_d(\alpha) \geq \frac{\alpha}{d+1}$.
When $d = 1$, it was shown by Abbot and Katchalski~\cite{AK79} that the optimal bound is $\beta_1(\alpha) = 1-\sqrt{1-\alpha}$.
Later, Kalai~\cite{Kal84} and Eckhoff~\cite{Eck85} proved the optimal bound for all dimention: $\beta_d(\alpha) = 1-(1-\alpha)^{1/(d+1)}$. See also~\cite{AK85} for a simple proof, which uses a set pair inequality~\cite{Alon85}.
In fact, the $\beta_d(\alpha)|\mathcal{F}|$ bound is only asymptotically tight and we can add a positive $o(|\mathcal{F}|)$ term.
Indeed, these results are all proved in the following exact form. 
\begin{theorem}[The fractional Helly Theorem]
Let $d,r$, and $n$ be positive integers such that $n>d+r$, and $\mathcal{F}$ is a family of $n$ convex sets in $\mathbb{R}^d$.
If more than $\binom{n}{d+1}-\binom{n-r}{d+1}$ of the $(d+1)$-tuples of the family $\mathcal{F}$ are intersecting, then $\mathcal{F}$ contains an intersecting subfamily of size at least $d+r+1$.
\end{theorem}

\subsection{Cartesian product of convex sets}
While the convex sets in $\mathbb{R}^d$ has Helly number $d+1$, a proper subfamily of it may have a smaller Helly number.
We observe the following fractional Helly type statement for such families of convex sets in $\mathbb{R}^d$.
\begin{proposition}\label{thm:frachel_smallhelly}
Let $d$ and $r$ be positive integers such that $d+1 \geq r \geq 2$.
Then there exist $c_{d+1,r} \in [0,1)$ and $\beta_{d+1,r}:(c_{d+1,r},1]\to(0,1]$ such that for every (possibly infinite) family $\mathcal{F}$ of convex sets in $\mathbb{R}^d$ with Helly number $r$ the following holds: for every finite subfamily $\mathcal{G}$ of $\mathcal{F}$ and $\alpha \in (c_{d+1,r},1]$, if at least $\alpha\binom{|\mathcal{G}|}{r}$ of the $r$-tuples are intersecting, then $\mathcal{G}$ contains an intersecting subfamily of size $\beta_{d+1,r}(\alpha)|\mathcal{G}|$.
\end{proposition}
Note that the fractional Helly theorem is the special case $r=d+1$ with $c_{d+1,d+1}=0$ and $\beta_{d+1,d+1}(\alpha)= 1-(1-\alpha)^{1/(d+1)}$. When $r < d+1$,  the assumption that $\mathcal{F}$ has Helly number $r$ above is necessary.
Consider for instance a family of hyperplanes in $\mathbb{R}^d$ in general position.
Here, hyperplanes in $\mathbb{R}^d$ are {\em in general position} if for every collection of $m \leq d$ hyperplanes, their intersection is a $(d-m)$-flat and every $d+1$ or more hyperplanes have no point in common.
It is clear that there are no intersecting subfamily of size larger than $d$ while every $d$-tuple is intersecting.

One of the most well-studied examples of families of convex sets with small Helly number is the family of Cartesian products of convex sets.
For positive integers $t$ and $d$, let $\mathcal{F}_{t,d}$ be the family of all convex sets of the form $A_1\times A_2\times\cdots\times A_t \subset \mathbb{R}^d\times\mathbb{R}^d\times\cdots\times\mathbb{R}^d \simeq \mathbb{R}^{td}$, where each $A_i$ is a convex set in $\mathbb{R}^d$.
Note that $\mathcal{F}_{t,d}$ has Helly number $d+1 \leq td+1$, hence Proposition~\ref{thm:frachel_smallhelly} applies. In particular, $\cF_{t,1}$ is the family of axis-aligned boxes.  A classical result of Eckhoff~\cite{Eck88} in 1988 provided a quantitatively optimal version of Proposition~\ref{thm:frachel_smallhelly} for axis-aligned boxes $\cF_{t,1}$ with $c_{td+1,2}=1-\frac{1}{t}$ and $\beta_{td+1,2}(\alpha)=1-\sqrt{t(1-\alpha)}$.

Answering a question raised by B\'ar\'any and Kalai~\cite[Problem 3.7]{BK21} and independently Lew~\cite{Lew}, we generalize Eckhoff's theorem on axis-aligned boxes to higher dimension, proving a quantitatively optimal fractional Helly type theorem for Cartesian products of convex sets in all dimensions.

\begin{theorem}\label{thm:frachel_genbox}
Let $t$ and $d$ be positive integers.
Let $\mathcal{F}$ be a finite subfamily of $\mathcal{F}_{t,d}$.
For every $\alpha \in (1-\frac{1}{t^d},1]$, if at least $\alpha\binom{|\mathcal{F}|}{d+1}$ of the $(d+1)$-tuples are intersecting, then $\mathcal{F}$ contains an intersecting subfamily of size $(1-(t^d(1-\alpha))^{1/(d+1)})|\mathcal{F}|$.
\end{theorem}

In addition, we provide a construction that shows the condition $\alpha > 1-\frac{1}{t^d}$ in the assumption and the fraction $1-(t^d(1-\alpha))^{1/(d+1)}$ in the conclusion cannot be improved.
Indeed, we prove both the above theorem and the construction in their exact forms, see Theorem~\ref{thm:leray_intersection} and Theorem~\ref{thm:frachel_genbox2}.

\subsection{Intersection of $d$-Leray complexes}
The assertions of Helly's theorem and its colorful~\cite{KM05} and fractional~\cite{Kal84} generalizations hold for more general set-systems, which satisfy certain topological conditions.
One important such topological condition is the hereditary homological dimension ($d$-Lerayness) of the simplicial complex (nerve) that reflects all the intersection pattern of a given family.

An {\em abstract simplicial complex} $X$ with ground set $V$ is a collection of subsets of $V$ that is closed under taking subsets: that is, if $\sigma$ and $\tau$ are subsets of $V$ such that $\sigma \subset \tau \in X$, then $\sigma \in X$.
Each $\sigma \in X$ is called a {\em face}, or a {\em simplex} of $X$ and each $v\in V$ is called a {\em vertex} of $X$.
For $W \subset V$, we denote by $X[W]$ the subcomplex of $X$ induced by $W$, that is,
\[X[W] := \{\sigma \in X: \sigma \subset W\}.\]
When $V$ is finite, $X$ is {\em $d$-Leray} if the reduced homology groups in dimension at least $d$ are trivial for all induced subcomplexes, that is, $\tilde{H}_i(X[W];\mathbb{Q}) = 0$ for all $i \geq d$ and $W \subset V$.
Given a family $\mathcal{F}$ of non-empty sets, the {\em nerve} of $\mathcal{F}$ is the simplicial complex
\[N(\mathcal{F}) := \{\mathcal{H} \subset \mathcal{F}: \mathcal{H} = \varnothing\text{ or }\mathcal{H}\text{ is intersecting}\}.\]
It is a well-known fact that the nerve of a finite family of convex sets in $\mathbb{R}^d$ is $d$-Leray.
See \cite{Tan13} for an overview on Helly type theorems and Leray complexes.

It is straightforward from the definition of the $d$-Lerayness that given a $d$-Leray complex $K$, every vertex subset $W$ with $\binom{W}{d+1} \subset K$ is a face of $K$.
This shows that every family of non-empty sets has Helly number at most $d+1$ if the nerve of any finite subfamily is $d$-Leray.
The optimal fractional Helly theorem for convex sets in $\mathbb{R}^d$ also can be generalized to $d$-Leray complexes~\cite[Theorem 13]{AKMM02}:
if $K$ is a $d$-Leray complex on $n$ vertices, then 
\begin{align}\label{leray_frachel}
\dim K < d+r \implies f_d(K) \leq \binom{n}{d+1} - \binom{n-r}{d+1},
\end{align}
where $f_d(K)$ is the number of $d$-dimensional faces of $K$ and the dimension $\dim K $ of the complex $K$ is the size of largest face minus $1$.

We study the following natural question on the intersection of $d$-Leray complexes.
Write $[t] := \{1,2,\ldots,t\}$. Let $K_i$, $i\in[t]$, be $n$-vertex $d$-Leray complexes and consider their intersection $\bigcap_{i\in [t]} K_i$. 
If $\bigcap_{i\in [t]} K_i$ has dimension less than $d+r$, how many $d$-dimensional faces can it have? Consider the following construction: 
For $r\geq (t-1)d$ and $n> r+d$, partition a set $V$ of $n$ vertices into $V_0, V_1,\dots, V_t$ with size $r-(t-1)d, n_1,\dots, n_t$ respectively and let any set $U\subseteq V$ satisfying $|U\cap V_i|\leq d$ for all $i\in [t]$ be a face.
If we optimize over the choices of $n_1,\dots, n_t$ with $n_1+\dots + n_t = n-r+(t-1)d$, we obtain a simplicial complex with 
$$g_d(n,t,r) = \binom{n}{d+1} - \min_{n_1,\dots, n_t} \sum_{i\in [t]} \binom{n_i}{d+1}$$
$d$-dimensional faces. The value of $g_d(n,t,r)$ is attained when all $n_1,\dots, n_t$ are as close as possible, i.e. $|n_i-n_j|\leq 1$ for all $i,j\in [t]$.
We will see later that such a simplicial complex is an intersection of $t$ many $d$-Leray complexes and has the dimension exactly $d+r-1$.
Denote by $K_d(n,t,r)$ the above complex with the optimal choice of $n_1,\dots, n_t$. 

We show that given the dimension of the intersection of $d$-Leray complexes, $g_d(n,t,r)$ bounds the number of $d$-dimensional faces from above. Thus, the construction $K_d(n,t,r)$ is an extremal example.

\begin{theorem}\label{thm:leray_intersection}
Let $d$, $r$, $t$, and $n$ be positive integers such that $n > d+r$ and $r \geq (t-1)d$.
Let $V$ be a set of $n$ vertices, and let $K_1,\ldots,K_t$ be $d$-Leray complexes on $V$. If the intersection $K = \bigcap_{i=1}^{t}K_i$ has dimension $\dim K < d+r$, then 
\[f_d(K) \leq g_d(n,t,r).\]
Moreover, the above upper bound is best possible, that is, there exists such $K$ that satisfies the equality.
\end{theorem}

\begin{rmk}\label{rmk:Hellynumber}
Theorem~\ref{thm:frachel_genbox} in fact follows from the above theorem.
The nerve of Cartesian products $\prod_{i\in [t]}A_i$ in $\prod_{i\in [t]}\mathbb{R}^d$ is the intersection of the nerves of their $i$-th coordinate projections. Each nerve of the projection is a nerve of convex sets in $\mathbb{R}^d$, hence a $d$-Leray complex. 
If we rewrite the above exact result into an asymptotic form,  $\alpha \binom{n}{d+1}$ being equal to $g_d(n,t,r)+1$ implies that 
$d+r+1 \geq (1- (t^d (1-\alpha))^{1/(d+1)}) n$.
Therefore, Theorem~\ref{thm:frachel_genbox} immediately follows from Theorem~\ref{thm:leray_intersection}.
\end{rmk}

\subsection{Rich family of extremal configurations}
A classical construction showing the sharpness of the fractional Helly theorem in~\eqref{leray_frachel} is as follows: considering $r$ copies of $\mathbb{R}^d$ and $n-r$ hyperplanes in general position. For the sharpness of Theorem~\ref{thm:leray_intersection}, one can consider the complex $K_d(n,t,r)$ (see Remark~\ref{rmk:tightness}).
 
 Unless $r$ is very close to $n$, both of those constructions contain a large $\Omega(n)$-size subfamily with no intersecting $(d+1)$-tuples.
It is thus natural to ask if we can achieve a better bound when stepping away from such families, that is, when we consider families that contain no large subfamilies without intersecting $(d+1)$-tuples. Note that this is equivalent to imposing a \emph{$(p,d+1)$-condition}, i.e. every $p$-tuple contains an intersecting $(d+1)$-tuple. The $(p,d+1)$-condition when $p$ is a constant is well-studied; it is related to the $(p,q)$-theorem, see concluding remark for more on this.

It is very tempting to conjecture that the quantitative bounds in the fractional Helly theorem or in Theorems~\ref{thm:frachel_genbox} and~\ref{thm:leray_intersection} can be further improved for $n$-element families with $(o(n),d+1)$-condition, i.e.~every linear-size subfamily contains an intersecting $(d+1)$-tuple. Indeed, such a phenomenon has occurred in Ramsey-Tur\'an theory in which the maximum edge-density of a graph without a fixed size clique can be significantly lowered when imposing an additional sublinear independence number condition, see e.g.~\cite{LRSS21} and references therein.

Our next result, somewhat to our own surprise, shows that this speculation is not true in a strong sense. Let alone an $(o(n),d+1)$-condition, even a $(p,q)$-condition with constant $p$ is not sufficient. For given $\alpha \in (0,1]$, we construct families $\mathcal{F}$ in $\mathbb{R}^d$ with $\alpha \binom{|\mathcal{F}|}{d+1}$ intersecting $(d+1)$-tuples with $(C_{\alpha,d},d+1)$-condition which have the same bound as in~\eqref{leray_frachel}. Here $C_{\alpha,d}$ is some constant depending on $\alpha$ and $d$.
Our construction shows that the optimal bound in fractional Helly theorem is stable in the sense that there is a rich family of different extremal configurations.

\begin{theorem}\label{thm:d-repr_example}
Let $d,r,n$ be positive integers with $n\geq d+r$.
There exists a family $\mathcal{F}$ of $n$ convex sets in $\mathbb{R}^d$ satisfying the following, where $K$
is the nerve of $\mathcal{F}$.
\begin{itemize}
\item Any subfamily of more than $d+\frac{n-d}{r+1}$ sets in $\mathcal{F}$ contains an intersecting $(d+1)$-tuple,
    \item the maximal size of an intersecting subfamily of $\mathcal{F}$ is at most $d+r$, i.e. $\dim K< d+r$,
    \item $\mathcal{F}$ contains $\binom{n}{d+1} - \binom{n-r}{d+1}$ intersecting $(d+1)$-tuples, i.e.
    the equality in  \eqref{leray_frachel} holds.
\end{itemize}
\end{theorem}
Note that for given $\alpha \in (0,1]$,  the equation $\alpha\binom{n}{d+1}= \binom{n}{d+1} - \binom{n-r}{d+1}$ implies 
$$r = \beta_d(\alpha) n + o(n) = (1- (1-\alpha)^{1/(d+1)})n+o(n),$$
and the bound $p=d+\frac{n-d}{r+1}+1$ on $(p,q)$-condition is at most a constant $C_{\alpha,d}$ depending on $\alpha$ and $d$.
Thus $(C_{\alpha,d},q)$-condition does not improve the bound $\beta_d(\alpha)= 1- (1-\alpha)^{1/(d+1)}$ any further.

We also give a construction for the sharpness of Theorem~\ref{thm:frachel_genbox}, showing that the additional $(C_{\alpha,t,d},d+1)$-condition has negligible effect on the quantitative bound of fractional Helly theorem for Cartesian product of convex sets.

\begin{theorem}\label{thm:frachel_genbox2}
Let $d,t,r,n$ be positive integers with $n\geq d+r$ and $r>(t-1)d$.
There exists a family $\mathcal{F}\subseteq \mathcal{F}_{t,d}$ of $n$ convex sets satisfying the following.
\begin{itemize}
\item any subfamily of at least $d + \frac{n- t(d-1)}{r-(t-1)d}$ sets in $\mathcal{F}$ contains an intersecting $(d+1)$-tuple,
    \item the maximal size of an intersecting subfamily of $\mathcal{F}$ is at most $d+r$,
    \item the number of intersecting $(d+1)$-tuples is exactly $g_d(n,t,r)$.
\end{itemize}
\end{theorem}

Again, for given $\alpha\in (1- \frac{1}{t^d},1]$, if $\mathcal{F}$ contains $\alpha\binom{n}{d+1} = g_d(n,t,r)$ intersecting $(d+1)$-tuples, then we have
$$r = (1- (t^d(1-\alpha))^{1/(d+1)})n+o(n)$$
and the bound $d + \frac{n-t(d-1)}{r-(t-1)d}+1$ is at most a constant $C_{\alpha,t,d}$ depending only on $\alpha, d, t$. Hence, $(C_{\alpha,t,d},q)$-condition does not improve the bound $1- (t^d(1-\alpha))^{1/(d+1)}$ any further. 

\medskip

\noindent\textbf{Organization.} The rest of the paper will be organized as follows.
In Section~\ref{sec:Leray_intersect}, we prove Theorem~\ref{thm:leray_intersection}, and provide an example that shows the tightness of Theorems~\ref{thm:frachel_genbox} and ~\ref{thm:leray_intersection}.
In Section~\ref{sec:construction}, we give constructions to prove Theorems~\ref{thm:d-repr_example} and~\ref{thm:frachel_genbox2}.
We finish with a few concluding remarks which include some open problems.

\section{Intersection of $d$-Leray complexes}\label{sec:Leray_intersect}
The basic idea to prove Theorem~\ref{thm:leray_intersection} is applying the fractional Helly theorem for $d$-Leray complexes repeatedly. 
The following proposition is a useful tool for our optimization. It is an easy consequence of Karamata's inequality (see, e.g., \cite{Kar32}). 
\begin{proposition}\label{prop:jensen}
Let $t$ and $k$ be positive integers. Let $x_1,\dots, x_t$ be nonnegative integers with $x=\sum_{i\in [t]} x_i$. 
If $x \geq tq + s$ for some nonnegative intger $q$ and $s$ with $0 \leq s < t$, then \[\sum_{i\in[t]} \binom{x_i}{k} \geq s \binom{q+1}{k} + (t-s)\binom{q}{k}.\]
\end{proposition}

Before we prove Theorem~\ref{thm:leray_intersection}, we remark that, when $d = 1$ and $r = t-1$, we have an obvious upper bound from a famous result in extremal graph theory, the so-called Tur\'{a}n's theorem.

\begin{theorem}[Tur\'{a}n's theorem, \cite{Tur41}]\label{thm:turan}
Let $q$, $s$, $t$ and $n$ be positive integers such that $n = tq + s$ and $0 \leq s < t$.
For every graph on $n$ vertices with no clique of size $t+1$, the number of edges is at most
\[ g_1
(n,t,t-1)=\binom{n}{2} - s\binom{q+1}{2} - (t-s)\binom{q}{2} .\]
\end{theorem}

Now we prove the optimal fractional Helly theorem for the intersection of $d$-Leray complexes.
\begingroup
\def\thetheorem{\ref{thm:leray_intersection}}
\begin{theorem}
Let $d$, $r$, $t$, and $n$ be positive integers such that $n > d+r$ and $r \geq (t-1)d$.
Let $V$ be a set of $n$ vertices, and let $K_1,\ldots,K_t$ be $d$-Leray complexes on $V$. If the intersection $K = \bigcap_{i=1}^{t}K_i$ has dimension $\dim K < d+r$, then 
\[f_d(K) \leq g_d(n,t,r).\]
\end{theorem}
\addtocounter{theorem}{-1}
\endgroup
\begin{proof}
If $t = 1$, then the statement follows from \eqref{leray_frachel}. Thus we may assume $t > 1$.
Adding all $i$-tuples with $i\leq d$ does not change the $d$-Lerayness and affect the conclusion of the statement, we may assume that every $i$-tuple in $V$ forms a face for $i\leq d$.

From each $K_i$, we will take a set $F_i$ of $(d+1)$-tuples that are not faces so that $F_i$'s are mutually disjoint.
Then we obtain an upper bound $f_d(K) \leq \binom{n}{d+1} - \sum_{i}|F_i|$.
Let $\bar{f_d}(\cdot)$ be the number of $(d+1)$-tuples that are not faces. 

Note that we may assume $\dim K \geq d$. Then we have $\dim K_i \geq \dim K \geq d$ for each $i \in [t]$.
Suppose $\dim K_1 = r_1 + d-1$ for some $r_1 > 0$.
Then, by~\eqref{leray_frachel}, we have $\bar{f_d}(K_1) \geq \binom{n-r_1}{d+1}$.
Let $F_1$ be the set of all $(d+1)$-tuples contributing to $\bar{f_d}(K_1)$, and let $W_1$ be a $(d+r_1-1)$-dimensional face in $K_1$.
For $1 \leq j < t$, we define $F_{j+1}$ and $W_{j+1}$ inductively as follows.

Assume that $W_j$ is a $(d+r_j-1)$-dimensional face of $K_j$. Let 
$$r_{j+1} =  \dim K_{j+1}[W_j] -(d-1).$$ 
Then we have \[\bar{f_d}(K_{j+1}[W_j]) \geq \binom{|W_j|-r_{j+1}}{d+1} = \binom{d+r_j-r_{j+1}}{d+1}.\]
Indeed, this is trivial if $r_{j+1}=0$ and follows from \eqref{leray_frachel} if  $r_{j+1}>0$. 

Let $F_{j+1}$ be the set of all $(d+1)$-tuples contributing to $\bar{f_d}(K_{j+1}[W_j])$.
Note that $F_{j+1}$ is disjoint from $\bigcup_{1 \leq i \leq j} F_i$, since $W_i$ is a simplex in $K_i$ for each $i\in [j]$.
Let $W_{j+1}$ be a largest face in $K_{j+1}$ which is a $(d+r_{j+1}-1)$-dimensional face. As we assume every $d$-tuple forms a face, we have $r_{j+1}\geq 0$. Repeating this defines $F_1,\dots, F_t$ and $W_1,\dots, W_t$.

By collecting all $F_j$'s, we obtain
\begin{align}\label{toapplyjensen} \bar{f_d}(K) \geq \binom{n-r_1}{d+1} + \binom{d+r_1-r_2}{d+1} + \binom{d+r_2-r_3}{d+1} + \cdots + \binom{d+r_{t-1}-r_t}{d+1}.\end{align}
Note that by definition $W_t \in K$, thus we have 
$$d+r_t -1 =\dim K[W_t]\leq \dim K < d+r,$$ 
implying that $r_t \leq r$. Consequently, 
\[(n-r_1)+ (d+r_1-r_2) +\dots + (d+r_{t-1}-r_t) \geq n-r + (t-1)d.\] 
By  Proposition~\ref{prop:jensen}, $\bar{f_d}(K)$ is at least $\sum_{i\in [t]} \binom{n_i}{d+1}$ where $\sum_{i\in [t]} n_i = n-r+(t-1)d$ and $|n_i-n_j|\leq 1$ for all $i\neq j$. This completes the proof.
\end{proof}

\begin{rmk}\label{rmk:tightness}
To see the sharpness of Theorem~\ref{thm:leray_intersection}, take the complex $K_{d}(n,t,r)$ and
let $K_i$, $i\in[t]$, be the simplicial complex where a set $U\subset V$  is a face of $K_i$ if and only if $|U\cap V_i|\leq d$.
This complex $K_i$ can be expressed as the nerve of the family $\mathcal{F}_i$ of convex sets in $\mathbb{R}^d$ consisting of  $|V|-|V_i|$ copies of $\mathbb{R}^d$ corresponding to the vertices outside $V_i$ and $|V_i|$ hyperplanes in general position in $\mathbb{R}^d$ corresponding to the vertices in $V_i$. 
For each $v\in V$ , let $\psi_i(v)$ be the convex set in $\mathcal{F}_i$ corresponding to the vertex $v$. Let $\mathcal{F}$ be the collection of $n$ Cartesian products $\prod_{i\in [k]} \psi_i(v)$, $v\in V$. Then the nerve of $\mathcal{F}$ is $K_d(n,t,r) = \bigcap_{i\in [t]} K_i$. A maximal face contains all vertices in $V_0$ and $d$ vertices from each $V_i$ with $i>0$. The dimension of $K_d(n,t,r)$ is exactly $d+r-1$. This shows that the upper bound in Theorem~\ref{thm:leray_intersection} is tight.
\end{rmk}

\section{No large subfamily without intersecting $(d+1)$-tuples}\label{sec:construction}

\subsection{Extremal example of convex sets in $\mathbb{R}^d$}\label{subsec:d-repr}
We now prove Theorem~\ref{thm:d-repr_example} by constructing a family of convex sets in $\mathbb{R}^d$ satisfying the conditions in the theorem. For this, we need the following lemma.
\begin{lemma}\label{lemma:hyperplanes}
Let $d$, $n$, and $r$ be positive integers such that $n \geq d+r$.
Then there exist $n$ convex sets $A_1,A_2,\ldots,A_n$ in $\mathbb{R}^d$ such that for every $1 \leq i_1 < i_2 < \cdots < i_{d+1} \leq n$, the intersection $\bigcap_{j =1}^{d+1}{A_{i_j}}$ is non-empty if and only if $i_{d+1} - i_d \leq r$.
\end{lemma}

Before we prove Lemma~\ref{lemma:hyperplanes}, we first present the proof of Theorem~\ref{thm:d-repr_example} based on Lemma~\ref{lemma:hyperplanes}.
\begin{proof}[Proof of Theorem~\ref{thm:d-repr_example}]
Let $\mathcal{H}$ be the hypergraph on $[n]$ such that for every $1 \leq a_1 < a_2 < \cdots < a_{d+1} \leq n$, the set $\{a_1,a_2,\ldots,a_{d+1}\}$ is an edge in $\mathcal{H}$ if and only if $a_{d+1} - a_d \leq r$.
By Lemma~\ref{lemma:hyperplanes}, there exists a family $\mathcal{F}$ of $n$ convex sets in $\mathbb{R}^d$ whose $(d+1)$-intersection hypergraph is isomorphic to $\mathcal{H}$, i.e. the collection of intersecting $(d+1)$-tuples of convex sets in $\mathcal{F}$ forms a hypergraph isomorphic to $\mathcal{H}$.
We claim that the family $\mathcal{F}$ satisfy the conditions in the statement.
\begin{claim}\label{claim:f_d}
$\displaystyle{|E(\mathcal{H})| = \binom{n}{d+1} - \binom{n-r}{d+1}}$.
\end{claim}
\begin{poc}
It suffices to show that $\Big|\binom{[n]}{d+1}-E(\mathcal{H})\Big| = \binom{n-r}{d+1}.$
For each $A\in \binom{[n]}{d+1}-E(\mathcal{H})$, we know that the largest number in $A$ is bigger than the second largest number plus $r$. Let $h(A)$ be the set we obtain from $A$ by replacing the largest number $a$ in $A$ with $a-r$. 
Then it is easy to see that $h$ is a bijection from $\binom{[n]}{d+1}-E(\mathcal{H})$ to $\binom{[n-r]}{d+1}$, proving the claim.
\end{poc}

\begin{claim}\label{claim:dim}
The maximal size of a clique of $\mathcal{H}$ is $d+r$, and the maximal size of an independent set of $\mathcal{H}$ is $d + \lfloor\frac{n-d}{r+1}\rfloor$.
\end{claim}
\begin{poc}
Let $W$ be a set of vertices $w_1,w_2,\ldots,w_k$ of $\mathcal{H}$ such that $1 \leq w_1 < w_2 < \cdots < w_k \leq n$.

We first prove the maximum size of a clique of $\mathcal{H}$ is $d+r$. 
If $W$ forms a clique, then $\{w_1,\dots,w_d,w_k\}$ must be an edge in $\mathcal{H}$. Thus $w_d< w_{d+1}< \dots < w_k \leq w_d+r$. This implies $k-d\leq r$. Hence, every clique of $\mathcal{H}$ has size at most $d + r$.
On the other hand, the equality holds $\{1,2,\ldots,d+r\}$ is a clique of size $d+r$.

Next, we prove that the maximum size of an independent set of $\mathcal{H}$ is $d+s$, where $s=\lfloor\frac{n-d}{r+1}\rfloor$. 
Again, 
by the definition of $\mathcal{H}$, $W$ is independent in $\mathcal{H}$ if and only if $w_j - w_i > r$ for every $d \leq i < j \leq k$.
Therefore, if $W$ is an independent set of $\mathcal{H}$, then we have $n\geq w_k \geq (r+1)(k-d)+ w_{d}$. Since $w_{d}\geq d$, we have 
$|W|=k \leq d+ s$.
On the other hand, 
$$\{1,\dots, d-1,d, d+(r+1), d+2(r+1),\dots, d+ s(r+1)\}$$ 
is an independent set of size $d+ s$.
\end{poc}

Note that, by Helly's theorem for convex sets in $\mathbb{R}^d$, a clique of $\mathcal{H}$ of size at least $d+1$ corresponds to an intersecting subfamily of $\mathcal{F}$.
Thus, Claim~\ref{claim:f_d} and Claim~\ref{claim:dim} shows that $\mathcal{F}$ satisfies the conditions of the statement.
This completes the proof.
\end{proof}

Now we give a proof of Lemma~\ref{lemma:hyperplanes} via an explicit construction.
\begin{proof}[Proof of Lemma~\ref{lemma:hyperplanes}]
We construct $A_k$ inductively.
Let $e_d = (0,0,\ldots,0,1)\in\mathbb{R}^d$.
First, let $u_1,u_2,\ldots,u_d$ be vectors in $\mathbb{R}^d$ such that every $d$ of the vectors $u_1,u_2,\ldots,u_d,e_d$ are linearly independent in $\mathbb{R}^d$.
Let $H_1,H_2,\ldots,H_d$ be $d$ hyperplanes in $\mathbb{R}^d$ where $u_i$ is the normal vector of $H_i$ for each $i\in [d]$. Note that $H_1,\dots, H_d$ are 
in general position and every $d$ hyperplanes in general position in $\mathbb{R}^d$ meet at exactly one point.

We start by letting $A_i = H_i$ for each $i \in [d]$.
Given a point $x = (x_1,x_2,\ldots,x_d) \in \mathbb{R}^d$, denote by $\pi_d(x)$ the $d$-th coordinate of $x$, that is, $\pi_d(x) = x_d$.
Let $k \geq d$.  Suppose we have found hyperplanes $H_{d+1},\ldots,H_k$ having normal vectors $u_{d+1},\ldots,u_k$ respectively and convex sets $A_{d+1}, \ldots, A_k$ such that the following hold.
\begin{enumerate}[(i)]
\item Every $d$ of the vectors $u_1,u_2,\ldots,u_k,e_d$ are linearly independent. In particular, the hyperplanes $H_1,H_2,\ldots,H_k$ are in general position in $\mathbb{R}^d$.
\item For each $d < i \leq k$, there exists a positive real number $s_i$ such that 
		\[A_i = \bigcup_{0 \leq s \leq s_i} \left(H_i-s\cdot e_d\right),\]
	  where $H_i-s\cdot e_d = \{z-(0,0,\ldots,0,s): z \in H_i\}$.
	  That is, $A_i$ is the region bounded by two parallel hyperplanes $H_i-s_i\cdot e_d$ and $H_i$.
	  Thus, for each $\sigma\in\binom{[k]}{d}$, the intersection $A_\sigma = \bigcap_{i \in \sigma}A_i$ is compact. 
\item $\bigcap_{j =1}^{d+1}{A_{i_j}} \neq \varnothing$ if and only if $i_{d+1} - i_d \leq r$ for every $1 \leq i_1 < i_2 < \cdots < i_{d+1} \leq k$.
\item For every $\sigma, \tau \in \binom{[k]}{d}$,
\[\max\{\pi_d(x): x \in A_\sigma\} < \max\{\pi_d(x): x \in A_\tau\}\]
if the maximal element of $\sigma$ is smaller than the maximal element of $\tau$. Note that the above maximum exists as $A_{\sigma}$ and $A_{\tau}$ are both compact.
\end{enumerate}

Let $\ell = k+1-r$. We will define a hyperplane $H_{k+1}$ with normal vector $u_{k+1}$ and a closed convex set $A_{k+1}\subset \mathbb{R}^d$ that satisfies the above (i)--(iv).

For each $\sigma \in \binom{[k]}{d}$, let $t_\sigma = \max\{\pi_d(x):x\in A_\sigma\}$.
In order to ensure (iv), we take a vector $y \in \mathbb{R}^d$ with large $d$-th coordinate such that $\pi_d(y) > t_\sigma$ for any $\sigma \in \binom{[k]}{d}$. 
Let $H$ be the hyperplane passing through $y$ orthogonal to $e_d$:
\[H = \{z \in \mathbb{R}^d: \langle z-y,e_d\rangle = 0\},\] where $\langle a,b\rangle$ is the Euclidean inner product of two vectors $a, b \in \mathbb{R}^d$.
Note that for any point $z$ on the hyperplane $H$, we have $\pi_d(z) = \pi_d(y) > t_\sigma$ for $\sigma \in \binom{[k]}{d}$.

In order to ensure (iii) later, we wish to take $A_{k+1}$ so that it intersects $A_{\sigma}$ for $\sigma \in \binom{[k]}{d}$ if and only if $\sigma$ contains a number larger than equal to $\ell$. As (iv) holds for $H_1,\dots, H_k$, to check that (iii) remains true with $A_{k+1}$ added, we only have to consider $(k+1)$-tuples containing $k+1$ and  $\sigma \in \binom{[\ell]}{d}$.
For this, let 
\[ t = \min\{t_\sigma:\sigma\in\binom{[\ell]}{d} \text{ and } \ell\in\sigma\} \quad \text{and} \quad t' = \max\{t_\sigma:\sigma\in\binom{[\ell-1]}{d}\}. \]
When $\ell-1 < d$, let $t'  = -\infty$.
Then by definition and (iv), we have $t' < t$, and hence we can take a positive real number $s_{k+1}$ such that $t' < \pi_d(y)-s_{k+1} < t$.
Let \[A = \bigcup_{0\leq s\leq s_{k+1}}\left(H-s\cdot e_d\right)\]
Clearly, if $H$ and $A$ were to play the roles of $H_{k+1}$ and $A_{k+1}$, respectively,  then (ii), (iii), and (iv) hold. However, the normal vector of $H$ is $e_d$, so (i) would not hold in this case.
See Figure~\ref{sec4-fig1} for an illustration when $d=2, r=1, k=4$.

\begin{figure}[htbp]
\centering
\includegraphics[scale=0.67]{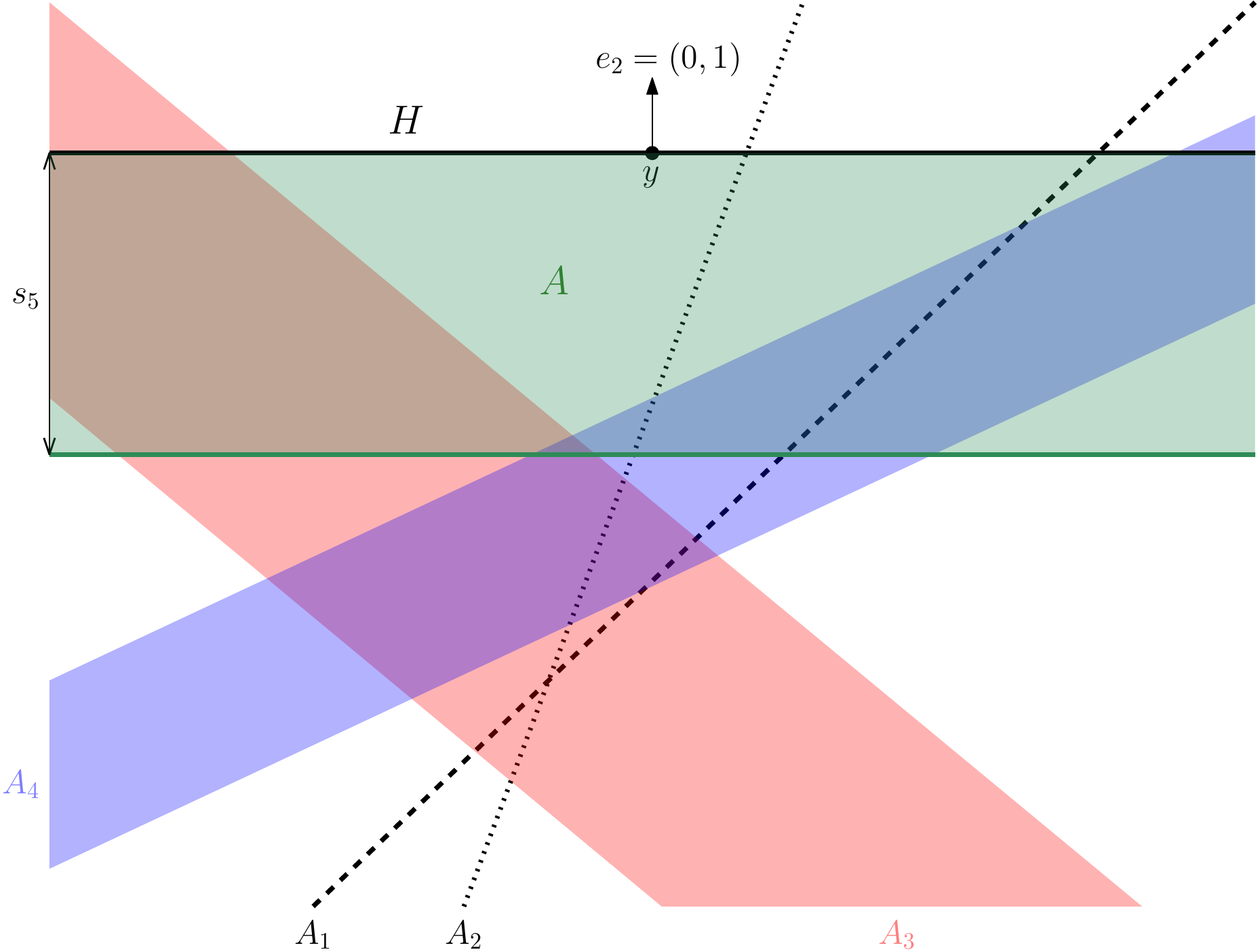}
\caption{If $A=A_{k+1}$ then conditions (ii), (iii), and (iv) hold.}
\label{sec4-fig1}
\end{figure}

In order to ensure (i) while keeping (ii)--(iv), we slightly perturb $H$ and $A$.
Take $u_{k+1} = e_d + (\epsilon_1,\epsilon_2,\ldots,\epsilon_d) \in\mathbb{R}^d$ for some $\epsilon_i \in \mathbb{R}$, $i \in [d]$, and let
\[H_{k+1} = \{z \in \mathbb{R}^d: \langle z-y,u_{k+1}\rangle = 0\}\;\;\;\text{and}\;\;\;A_{k+1} = \bigcup_{0\leq s \leq s_{k+1}}(H_{k+1}-s\cdot e_d).\]
Since $k$ is finite, we can take $\epsilon_i$'s with $|\epsilon_i|$ small enough, so that (i) holds while (ii)--(iv) still hold.  
See Figure~\ref{sec4-fig2} for an illustration of such modification to the example in Figure~\ref{sec4-fig1}.
\begin{figure}[htbp]
\centering
\includegraphics[scale=0.67]{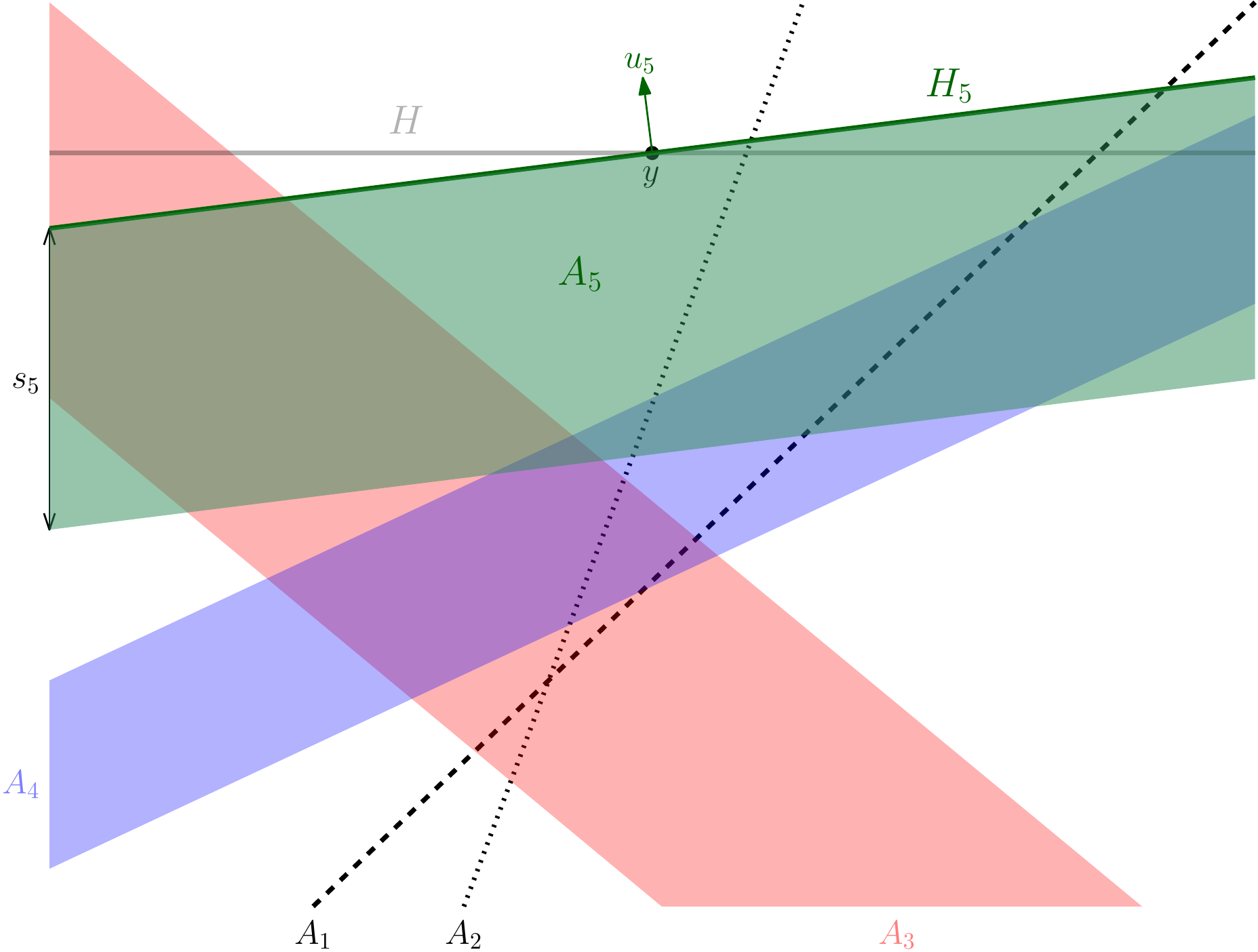}
\caption{$A_5$ is obtained by modifying $A$. It satisfies (i), (ii), (iii), and (iv).}
\label{sec4-fig2}
\end{figure}
Repeat this process until we obtain $A_n$. This completes the proof due to (iii).
\end{proof}

\subsection{Extremal example of Cartesian products of convex sets in $\mathbb{R}^d$}
Based on the construction in Section~\ref{subsec:d-repr}, we can also construct a subfamily of $\mathcal{F}_{t,d}$ that does not contain a large subfamily without intersecting $(d+1)$-tuples and satisfies an almost optimal fractional Helly property for $(d+1)$-tuples. We construct this subfamily to prove Theorem~\ref{thm:frachel_genbox2}.

\begin{proof}[Proof of Theorem~\ref{thm:frachel_genbox2}]
Let $n_1,\dots, n_t$ and $r_1,\dots, r_t$ be integers so that 
$n_1+\dots + n_t = n$ and $r_1 + \dots + r_t = r-(t-1)d$, and moreover
$ n_1\leq \dots \leq n_t \leq n_1+1$ and $r_1\leq \dots \leq r_t \leq r_1+1$.
In particular, then each $r_i$ is either $\left\lfloor\frac{r-(t-1)d}{t}\right\rfloor$ or $\left\lceil\frac{r-(t-1)d}{t}\right\rceil$ and we have
$$\sum_{i\in [t]} (n_i-r_i) = n-r+(t-1)d \enspace \text{ and }  \enspace
\left| (n_i-r_i) - (n_j-r_j) \right|\leq 1 \text{ for all } i,j\in [t].$$
For each $i\in [t]$, by Theorem~\ref{thm:d-repr_example}, there exists a family $\mathcal{F}'_i$ of $n_i$ convex sets in $\mathbb{R}^d$ such that the following holds.
\begin{enumerate}[(a)]
    \item Any family of more than $d + \frac{n_i -d}{r_i+1}$ convex sets in $\mathcal{F}'_i$ contains an intersecting $(d+1)$-tuple,
    \item the maximal size of an intersecting subfamily of $\mathcal{F}'_i$ is at most $d+r_i$, and
    \item $\mathcal{F}'_i$ contains exactly $\binom{n_i-r_i}{d+1}$ non-intersecting $(d+1)$-tuples.
\end{enumerate}
Recall that all $d$-tuples in the family in Theorem~\ref{thm:d-repr_example}, hence also those in $\mathcal{F}_i'$, are intersecting.
For $A\in \mathcal{F}'_i$, let $\pi^{-1}_i(A)$ be the set of all points $x\in (\mathbb{R}^{d})^t$ whose $i$-th coordinate projection $\pi_i(x)$ lies in $A$. 
 In other words, $\pi^{-1}_i(A)$ is the Cartesian product of $A\in \mathcal{F}'_i$ and $t-1$ copies of $\mathbb{R}^d$ so that $i$-th coordinate projection of the product is $A$.
Let 
$$\mathcal{F}= \bigcup_{ i\in [t]} \{ \pi^{-1}_i(A): A\in \mathcal{F}'_i\}.$$
Thus $\mathcal{F}$ has size $|\mathcal{F}|=\sum_{i\in[t]}|\mathcal{F}_i'|=n$.

Note that any subfamily $\mathcal{F}^*$ of $d+1$ sets in $\mathcal{F}$ are intersecting if 
not all of them come from the same $\mathcal{F}'_i$, i.e. $\mathcal{F}^* \not\subset  \{ \pi^{-1}_i(A): A\in \mathcal{F}'_i\}$ for any $i$. 
For each $i\in [t]$, as we know $n_i\leq \frac{n+t-1}{t}$ and $r_i\geq \frac{r - (t-1)d-(t-1)}{t}$, we have
\[ d+ \frac{n_i -d}{r_i+1}< d + \frac{n - t(d-1)}{r-(t-1)d}.\]
Consider a subfamily $\mathcal{F}^*$ of at least $d + \frac{n-t(d-1)}{r-(t-1)d}$ sets in $\mathcal{F}$. If not all of them come from the same $\mathcal{F}'_i$, then it contains an intersecting $(d+1)$-tuple. If all of them lie in $\{ \pi^{-1}_i(A): A\in \mathcal{F}'_i\}$, as $|\mathcal{F}^*| > d+ \frac{n_i -d}{r_i+1}$, (a) implies that this contains an intersecting $(d+1)$-tuple.

As any $d+1$ sets not coming from the same $\mathcal{F}'_i$ are intersecting, 
(b) implies that the largest family of intersecting sets has size at most $\sum_{i\in [t]} (d+r_i) = d + r$.
Moreover, as all non-intersecting $(d+1)$-tuples come from the same $\mathcal{F}'_i$, (c) implies that the number of non-intersecting $(d+1)$-tuples in $\mathcal{F}$ is exactly
$\sum_{i\in [t]} \binom{n_i-r_i}{d+1}$. Hence, the number intersecting $(d+1)$-tuples in $\mathcal{F}$ is 
$$\binom{n}{d+1} - \sum_{i\in [t]} \binom{n_i-r_i}{d+1} = g_d(n,t,r).$$
The equality holds as $n_i-r_i$ sums up to $n-r+(t-1)d$ and $|(n_i-r_i)- (n_j-r_j)|\leq 1$ for all $i,j\in [t]$. This proves the theorem.
\end{proof}

\section{Concluding remarks}\label{sec:rmk}
\subsection{Induced density vs forbidden blowup}
Given a family $\mathcal{F}$ of non-empty sets, the {\em colorful Helly number} is the maximal integer $m$ such that there exists finite subfamilies $\mathcal{F}_1,\ldots,\mathcal{F}_{m-1}$ of $\mathcal{F}$ such that each $\mathcal{F}_i$ is not intersecting and for all $A_i \in \mathcal{F}_i$, the intersection $\bigcap_{i\in[m-1]}A_i$ is non-empty.
By taking all $\mathcal{F}_i$'s identical, one can see that the colorful Helly number is always greater than or equal to the Helly number.
The colorful Helly theorem~\cite{Bar82}, which is another important generalization of Helly's theorem, asserts that the family of all convex sets in $\mathbb{R}^d$ has colorful Helly number $d+1$.
See also \cite{KM05} for the colorful Helly theorem for $d$-Leray complexes.

Kim~\cite{Kim17} recently gave an alternative way to show the robustness of the fractional Helly theorem, that is, $\beta_d$ tends to $1$ as $\alpha$ tends to $1$, using the colorful Helly theorem as a blackbox, see also \cite{BFMOP14, BGT21} for related works.
Improving the idea in~\cite{Kim17}, Holmsen \cite{Hol20} proved an extremal graph theoretic result, extending the work of Gy\'arf\'as, Hubenko and Solymosi~\cite{GHS02} to hypergraphs. It roughly states that 
dense hypergraphs with certain forbidden configuration must contain a linear-size clique. In particular, this result implies that the fractional Helly theorem can be derived from the colorful Helly theorem in a purely combinatorial way.
Such work is of great interest, as it has been one of the most fundamental question in Helly type problems to find sufficient combinatorial conditions to give some fractional Helly type result for abstract set-systems, see e.g. \cite{Mat04, Pat20, HL21, GHP21} for such works.

Write $K_t^{(2)}$ for the \emph{2-blowup} of $K_t$. That is, $K_t^{(2)}=K_{2,\ldots, 2}$ is the complete $t$-partite graph with each part of size 2. Holmsen's result~\cite{Hol20} for graphs reads as follows. 

\begin{theorem}[\cite{Hol20}] \label{thm:ind-blowup} 
  Let $G$ be a graph with no induced copy of $K_t^{(2)}$. If $G$ has positive $K_t$-density, then there is a linear-size clique in $G$.
\end{theorem}

We sketch a conceptually simpler proof of Theorem~\ref{thm:ind-blowup} using Szemer\'edi regularity lemma. We use some standard terminologies used in the literature, see, e.g. \cite{KSSS02}. Regularity lemma asserts that any graph $G$ can be decomposed into bounded number (say $k$) of parts so that most of the pairs of parts have certain pseudorandom properties. The reduced graph $R$ is a graph on vertex set $[k]$, where $ij$ is an edge if and only if the bipartite graph between $i$-th and $j$-th parts is random-like. By a standard embeddimg lemma, if $G$ has positive density of $K_t$, then the reduced graph $R$ contains a copy of $K_t$. An elementary argument shows that if all the parts corresponding to a copy of $K_t$ in $R$ are not dense enough, then one can embed $K_t^{(2)}$ using these parts in an obvious way. Thus, there must exist a part $V$ which is very dense, say having edge-density at least $1-1/(100t)$. By keep deleting vertices of low degree if exists, one can obtain a subset $V'\subseteq V$ with $|V'|\geq |V|/(100t)$ satisfying $\delta(G[V'])\geq (1- 1/(10t)) |V'|$. With this minimum degree condition, we can keep taking non-edges $u_iv_i$ from the common neighborhood of $u_1,v_1,\dots, u_{i-1}, v_{i-1}$ within $V'$ if exists. Note that the minimum degree condition ensures that the common neighborhood has size at least $|V'|/2$ if $i\leq t$. This either provides an induced copy of $K_t^{(2)}$ or a linear size clique, yielding the desired result.

\subsection{Fractional Helly properties of abstract set-systems}
Let $c \in [0,1)$ be a real number, $r > 1$ be an integer, and $\mathcal{F}$ be a family of non-empty sets.
We say $\mathcal{F}$ satisfies a {\em fractional Helly property for $r$-tuples over $(c,1]$} if it satisfies the following:
there exists $\gamma:(c,1]\to(0,1]$ such that for every $\alpha \in (c,1]$ and finite subfamily $\mathcal{H}$ of $\mathcal{F}$, if at least $\alpha\binom{|\mathcal{H}|}{r}$ of the $r$-tuples of $\mathcal{H}$ are intersecting, then $\mathcal{H}$ contains an intersecting subfamily of size at least $\gamma(\alpha)|\mathcal{H}|$.

In general, bounded Helly number may not guarantee any fractional Helly property.
To see this, consider the following concept.
Given a family $\mathcal{F}$ of sets, the {\em intersection graph} is a graph with vertex set $\mathcal{F}$ whose edge sets is the set of all intersecting pairs of $\mathcal{F}$.
Observe that any graph can be represented as the intersection graph of a family of non-empty sets with Helly~number~$2$.
Let $G$ be a graph on $V$ and let $\mathcal{C}$ be the set of all maximal cliques of $G$.
For each $v \in V$, let $C_v$ be the set of all maximal cliques of $G$ that contain $v$, that is, $C_v = \{C \in \mathcal{C}: v \in C\}$.
Observe that, for a vertex subset $W$, $\bigcap_{v \in W}C_v \neq \varnothing$ if and only if $W$ is a clique in $G$.
This implies that the family $\mathcal{G} = \{C_v: v \in V\}$ has Helly number $2$ and that the intersection graph of $\mathcal{G}$ is isomorphic to $G$.

Now, let $\mathcal{F}$ be a family of non-empty sets with Helly number $2$, whose intersection graph is the disjoint union of all possible graphs.
Then, for every integer $r \geq 2$, there is no $c \in [0,1)$ such that $\mathcal{F}$ satisfies the fractional Helly property for $r$-tuples over $(c,1]$.
For example, let $\mathcal{F}_m$ be a subfamily of $\mathcal{F}$ whose intersection graph is isomorphic to the complete $m$-partite graph $K_{m,m,\ldots,m}$.
Clearly, $\mathcal{F}_m$ consists of $m^2$ members and the maximal size of an intersecting subfamily of $\mathcal{F}$ is $m = o(|\mathcal{F}_m|)$.
On the other hand, there are exactly $\binom{m^2}{2}-m\binom{m}{2} = (1-o(1))\binom{|\mathcal{F}_m|}{2}$ intersecting pairs in $\mathcal{F}$.

The crucial reason why Proposition~\ref{thm:frachel_smallhelly} holds is that any family of convex sets in $\mathbb{R}^d$ satisfies the fractional Helly property for $(d+1)$-tuples over $(0,1]$.
In this point of view, we can reformulate Proposition~\ref{thm:frachel_smallhelly} in a slightly generalized form as follows.
\begin{proposition}\label{thm:frachel_smallhelly_abstract}
Let $k$ and $r$ be positive integers such that $k \geq r \geq 2$ and let $\alpha>0$. Then there exists $c_{k,r,\alpha} \in [0,1)$ such that the following holds:
for every finite family $\mathcal{F}$ of non-empty sets with Helly number $r$, if $\mathcal{F}$ satisfies the fractional Helly property for $k$-tuples over $(\alpha,1]$, then it satisfies the fractional Helly property for $r$-tuples over $(c_{k,r,\alpha},1]$.
\end{proposition}
\begin{proof}[Sketch of proof]
Let $G$ be the graph consisting of all 1-dimensional faces of $\mathcal{F}$. Note that by taking $c_{k,r,\alpha}\rightarrow 1$, the edge-density of $G$, hence also the $K_k$-density of $G$, tends to 1. We can then apply the fractional Helly property for $k$-tuples over $(\alpha,1]$.
\end{proof}

\subsection{The number of higher dimensional faces in $\mathcal{F}_{t,d}$}
A more general form of \eqref{leray_frachel} for higher dimension is known: for every $j \geq d$,
\[f_j(K) \leq \sum_{i=0}^{d}\binom{n-r}{i}\binom{r}{j+1-i}.\]
It would be interesting to find the tight upper bound for $f_j(K)$ for $j > d$, when $K$ is the nerve of a finite subfamily of $\mathcal{F}_{t,d}$.
When $d = 1$, it was shown in \cite{Eck88} that the construction in Remark~\ref{rmk:tightness} has the maximum number of intersecting $(j+1)$-tuples for every $j \geq d$. 
It was observed by Lew~\cite{Lew} that Eckhoff's arguments also applies to the intersection of $t$ many $1$-Leray complexes.
A natural guess is that the same holds for the intersections of $d$-Leray complexes for all $d>1$.

We conclude this discussion with a reformulation of a more general problem posed by B\'{a}r\'{a}ny and Kalai~\cite[Problem 3.8]{BK21}.
\begin{prob}[\cite{BK21}]\label{prob:hnumber2}
For every positive integers $d,k,n,r$ with $n > d+r, k \geq 2$, find the minimum integer $T(d,k,n,r)$ such that the following holds: for every finite family $\mathcal{F}$ of non-empty sets such that $\mathcal{F}$ has Helly number $2$ and the nerve of $\mathcal{F}$ is $d$-Leray, if more than $T(d,k,n,r)$ of the $(k+1)$-tuples of $\mathcal{F}$ are intersecting, then $\mathcal{F}$ contains an intersecting subfamily of size $d+r+1$.
\end{prob}

\subsection{Complexes with smaller dimensions}
Theorem~\ref{thm:leray_intersection} describes the fractional Helly property when $r \geq (t-1)d$.
It remains to investigate how $f_d(K)$ behaves when $r < (t-1)d$.
When $d = 1$, it was observed in \cite{Eck88} that the obvious upper bound follows from Theorem~\ref{thm:turan} is tight.
For the completeness of this remark, we include a proof.
\begin{theorem}
Let $K_1,K_2,\ldots,K_t$ be $1$-Leray complexes on $V$ with $|V| = n$, and let $K = \bigcap_{i\in[t]}K_i$.
If $\dim K < m \leq t$, then \[f_1(K) \leq \binom{n}{2} - s\binom{\left\lceil\frac{n}{m}\right\rceil}{2} - (m-s)\binom{\left\lfloor\frac{n}{m}\right\rfloor}{2},\]
where $s$ is a non-negative integer such that $0 \leq s < t$ and $n \equiv s$ modulo $m$.
Moreover, there exists such $K$ that satisfies the equality.
\end{theorem}
\begin{proof}
Let $G$ be a graph on $V$ whose edge set is the set of all $1$-dimensional faces of $K$.
Recall that, as mentioned in Remark~\ref{rmk:Hellynumber}, a vertex subset $W$ with $|W| \geq 2$ is a face in $K$ if and only if every pair in $W$ is a face in $K$.
Thus the condition $\dim K < m$ implies that $G$ has no clique of size $m + 1$.
Then upper bound immediately follows from Theorem~\ref{thm:turan}.

For the equality of the upper bound, consider a vertex partition $V = V_1 \cup V_2 \cup \cdots \cup V_m$ such that $|V_i| = \left\lceil\frac{n}{m}\right\rceil$ for $1 \leq i \leq s$ and $|V_i| = \left\lfloor\frac{n}{m}\right\rfloor$ for $s+1 \leq i \leq m$.
For each $i \in [m]$, let $K_i$ be the complex on $V$ such that $W \subset V$ is a face in $K_i$ if and only if $|W \cap V_i| \leq 1$.
Note that each $K_i$ can be express as the nerve of $|V|-|V_i|$ copies of $\mathbb{R}$ and $|V_i|$ distinct point on $\mathbb{R}$, and hence $K_i$ is $1$-Leray.
Let $K = \cap_{i\in[m]}K_i$. Then $K$ is a complex on $V$ such that $W \subset V$ is a face in $K$ if and only if $|W \cap V_i| \leq 1$ for each $i \in [m]$.
In particular, $f_1(K) = \binom{n}{2} - \sum_{i\in[m]}\binom{|V_i|}{2}$, as required.
\end{proof}

Note that for a constant $c$ the condition $\dim K < c$ in $d$-Leray complexes already provides an upper bound on the number of $d$-dimensional faces. Consider a $(d+1)$-uniform hypergraphs without any clique of size $c$, what would be the maximum possible number of edges in this hypergraph?
The answer to this question is the Tur\'{a}n number of the $(d+1)$-uniform complete hypergraph on $c$ vertices. In turn, this provides a natural upper bound on the number of $d$-dimensional faces in the intersection of $t$ many $d$-Leray complexes with $\dim K <(t-1)d$. Would this upper bound be tight?
As the Tur\'{a}n number of such a complete hypergraph is not even known for $d>1$, and it is expected that the extremal hypergraphs have more complicated structures than the constructions in Theorems~\ref{thm:d-repr_example} and \ref{thm:frachel_genbox2}. So, it seems unlikely that the bound supposedly provided from the hypergraph Tur\'{a}n number is tight for the family of convex sets in $\mathcal{F}_{t,d}$ or even in $\mathbb{R}^d$.

\subsection{$(p,q)$-condition for constant $p$}
In Theorem~\ref{thm:frachel_genbox2}, we constructed a family of convex sets in $\mathcal{F}_{t,d}$ that satisfies $(C_{\alpha,t,d},d+1)$-condition. We showed that $(C_{\alpha,t,d},d+1)$-condition does not provide any improvement on the fractional Helly property for $(d+1)$-tuples when $\alpha\binom{|\mathcal{F}|}{d+1}$ intersecting $(d+1)$-tuples are present for $\alpha \in (1- \frac{1}{t^d},1].$
However, would such $(p,q)$-condition extend the range of $\alpha$ for the fractional Helly property for $(d+1)$-tuples?
More precisely, does there exists $\delta < 1 - \frac{1}{t^d}$ such that the families of convex sets satisfying $(p,d+1)$-condition with bounded $p$ satisfy a fractional Helly property for $(d+1)$-tuples over $(\delta,1]$, instead of over $(1-\frac{1}{t^d},1]$?

This follows from the $(p,q)$-theorem for convex sets due to Alon and Kleitman~\cite{AK92}.
Roughly speaking, the $(p,q)$-theorem asserts that, if a family of convex sets satisfies a sufficiently strong intersection property, then the whole family can be ``pierced'' by few points.
Here is a precise statement of the $(p,q)$-theorem.
\begin{theorem}[$(p,q)$-theorem, \cite{AK92}]\label{thm:pq}
For every positive integers $d$, $p$, and $q$ with $p \geq q \geq d+1$, there exists a positive integer $N = N(d;p,q)$ such that the following holds:
for every finite family $\mathcal{F}$ of convex sets in $\mathbb{R}^d$ that satisfies the $(p,q)$-condition, there exists a set of at most $N$ points in $\mathbb{R}^d$ that meets all members of $\mathcal{F}$.
\end{theorem}
Note that Helly's theorem implies $N(d;d+1,d+1) = 1$.
See, for example, \cite{KST18,Rub21} for recent development related to the $(p,q)$-theorem.
See also \cite{AKMM02} for the $(p,q)$-theorem for abstract set-systems.

As a corollary of Theorem~\ref{thm:pq}, we can obtain the following.
\begin{cor}\label{cor:pq}
For every positive integers $d$, $p$, and $t$ with $p \geq d+1$, there exists a positive integer $M_t(d;p)$ such that the following holds:
for every finite family $\mathcal{F} \subset \mathcal{F}_{t,d}$ that satisfies the $(p,d+1)$-condition, there exists a set of at most $M_t(d;p)$ points in $\mathbb{R}^{td}$ that meets all members of $\mathcal{F}$.
\end{cor}
\begin{proof}
Let \[\mathcal{F} = \{A_{i,1}\times A_{i,2}\times\cdots\times A_{i,t} \subset \mathbb{R}^{td}: i\in[n]\},\] where $A_{i,j} \subset \mathbb{R}^d$ is a convex set in $\mathbb{R}^d$ for each $i \in [n]$ and $j \in [t]$.
Suppose $\mathcal{F}$ satisfies the $(p,d+1)$-condition.
Then, for each $j \in [t]$, the family $\mathcal{G}_j := \{A_{i,j}: i\in[n]\}$ also satisfies the $(p,d+1)$-condition, and hence there exists a set $S_i \subset \mathbb{R}^d$ of at most $N(d;p,d+1)$ points that intersects each member of $\mathcal{G}_j$.
Now, the set \[S := S_1 \times S_2 \times \cdots \times S_t \subset \mathbb{R}^d \times \mathbb{R}^d \times \cdots \mathbb{R}^d \simeq \mathbb{R}^{td}\] meets all members of $\mathcal{F}$.
This proves $M_t(d;p) \leq N(d;p,d+1)^t$.
\end{proof}

\begin{rmk}
Let $d$ and $p$ be positive integers such that $p \geq d+1$.
For every finite subfamily $\mathcal{F}$ of $\mathcal{F}_{t,d}$ satisfying $(p,d+1)$-condition, no matter how small $f_d(N(\mathcal{F}))$ is, Corollary~\ref{cor:pq} guarantees the existence of a subfamily of size at least $\frac{n}{M_t(d;p)}$.
Hence, this implies that such a family has a fractional Helly property for $(d+1)$-tuples over $(0,1]$. 

However, in order to conclude this, we might want to ask the following question: which values $\alpha$ in $(0,1]$ are realizable? In other words, for which $\alpha \in (0,1]$, there exists a family $\mathcal{F}\subseteq \mathcal{F}_{t,d}$ of $n$ sets satisfying $(p,d+1)$-condition while having $\alpha \binom{n}{d+1}$ intersecting $(d+1)$-tuples. For example, from the hypergraph Tur\'{a}n theorem, we know that such a family $\mathcal{F}$ must have at least $\binom{p-1}{d}^{-1} \binom{n}{d+1}$ intersecting $(d+1)$-tuples, hence $\alpha$ cannot be in the interval $(0,\binom{p-1}{d}^{-1})$. This bound is not tight even for general $(d+1)$-uniform hypergraphs.
Hence this leaves the following interesting question: if $\mathcal{F}$ satisfies $(p,d+1)$-condition, then how many intersecting $(d+1)$-tuples must it have? If our choice of $p$ ensures that this bound is smaller than $1-\frac{1}{t^d}$, then we would be able to say that the $(p,d+1)$-condition indeed ensures that the fractional Helly property for $(d+1)$-tuples holds over a larger range of $\alpha$.
\end{rmk}

Also, it is natural to ask the best possible $M_t(d;p)$.
This kind of questions has been studied for the case $d=1$, that is, for axis-aligned boxes: see, for example, \cite{CD20,CSZ18}.
We conclude the section with a well-known conjecture on $M_2(1;p)$.
\begin{question}[\cite{Weg65}]
Let $\mathcal{F}$ be a family of axis-aligned boxes in the plane.
If every $p$-tuple contains intersecting triple, then there exists a set of $2p-3$ points in the plane that meets all members of $\mathcal{F}$.
\end{question}
 
\section*{Acknowledgment}
The authors thank Alan Lew for introducing the problem about the intersection of $d$-Leray complexes.
The authors also thank Andreas Holmsen for suggesting an idea for the construction in Section~\ref{subsec:d-repr}.

\bibliographystyle{abbrv}
\bibliography{biblio}
\end{document}